\documentclass[12pt]{amsart}
\usepackage{amsmath,amsthm,amsfonts,latexsym,amssymb,amscd,color}
\usepackage{chemarr}
\newtheorem{thm}{Theorem}[section]
\newtheorem{cor}[thm]{Corollary}
\newtheorem{lm}[thm]{Lemma}

\newtheorem{clm}[thm]{Claim}
\newtheorem*{clm*}{Claim}
\theoremstyle{definition}
\newtheorem{df}[thm]{Definition}
\newtheorem{dfs}[thm]{Definitions}

\theoremstyle{remark}

\numberwithin{equation}{section}
\newcommand{\sprf}{\noindent{\it Proof.}} 
\newcommand{\sqed}{\hfill\rule{1.3mm}{3mm}\medskip}

\newcommand{\cproof}{\noindent{\it Proof of claim.}\ } 
\newcommand{\cqed}{\hfill\rule{1.3mm}{3mm}}

\newcommand{\wec}[1]{{\mathbf{#1}}}  
\newcommand{\m}[1]{{\mathbf{\uppercase{#1}}}}

\newcommand{\bd}{\begin{description}}
\newcommand{\ed}{\end{description}}
\newcommand{\lb}{\langle} 
\newcommand{\rb}{\rangle}

\let\phi=\varphi

\begin{document}

\title[Growth Rates of Algebras, II]{Growth Rates of Algebras, II:\\
\vspace{2mm} {\normalsize\rm Wiegold dichotomy}}

\author{Keith A. Kearnes}
\address[Keith Kearnes]{Department of Mathematics\\
University of Colorado\\
Boulder, CO 80309-0395\\
USA}
\email{Keith.Kearnes@Colorado.EDU}
\author{Emil W. Kiss}
\address[Emil W. Kiss]{
Lor\'{a}nd E{\"o}tv{\"o}s University\\
Department of Algebra and Number Theory\\
H--1117 Budapest, P\'{a}zm\'{a}ny P\'{e}ter sétány 1/c.\\
Hungary}
\email{ewkiss@cs.elte.hu}
\author{\'Agnes Szendrei}
\address[\'Agnes Szendrei]{Department of Mathematics\\
University of Colorado\\
Boulder, CO 80309-0395\\
USA}
\email{Agnes.Szendrei@Colorado.EDU}
\thanks{This material is based upon work supported by
the Hungarian National Foundation for Scientific Research (OTKA)
grant no.\ K77409, K83219, and K104251.
}
\subjclass{08A30 (08B05, 08B10)}
\keywords{Growth rate, Wiegold dichotomy, perfect algebra,
basic identity, cube term, parallelogram term, 
maximal subalgebra}

\begin{abstract}
We investigate the function $d_{\m a}(n)$, 
which gives the size of a least size
generating set for $\m a^n$, in the case where 
$\m a$ has a cube term.
We show that if $\m a$
has a $k$-cube term and $\m a^k$ is finitely generated,
then $d_{\m a}(n)\in O(\log(n))$ if $\m a$ is perfect
and  $d_{\m a}(n)\in O(n)$ if $\m a$ is imperfect.
When $\m a$ is finite, then one may replace 
``Big Oh'' with ``Big Theta'' in these estimates.
\end{abstract}

\maketitle

\section{Introduction}\label{intro_sec}
For an algebraic structure $\m a$, write $d_{\m a}(n)=g$
if $g$ is the least size of a generating set
for the direct power $\m a^n$. 
We call the function $d_{\m a}(n)$ the
\emph{growth rate} of $\m a$.
The study of this function originated in group
theory, and some of its history 
is surveyed 
in the preceding paper in this series, \cite{paper1}.
In the present paper we pursue a thread that may also be viewed
as originating in group theory, but is directly 
motivated by some of the results in \cite{paper1}.

James Wiegold proved in \cite{wiegold1} that 
the growth rate of a finite perfect group 
is logarithmic ($d_{\m a}(n)\in \Theta(\log(n))$),
and that 
the growth rate of a finite imperfect group 
is linear ($d_{\m a}(n)\in \Theta(n)$).
This result, herein called
\emph{Wiegold dichotomy}, was extended 
by Martyn Quick and Nik Ru\v skuc in \cite{quick_ruskuc}
to several kinds of
algebras that have underlying group structure.
Namely, Quick and Ru\v skuc 
showed that a finite algebra $\m a$ satisfies 
$d_{\m a}(n)\in \Theta(\log(n))$ if $\m a$
is a perfect 
ring, module, 
Lie algebra or $k$-algebra over a field $k$, and that 
$d_{\m a}(n)\in \Theta(n)$ if $\m a$
is an imperfect algebra of one of these types.

To put these results in a broader context,
call a term $t$ \emph{basic}
if it is a variable, a constant, 
or a function
symbol applied to variables and constants.
Call an identity $s\approx t$ basic if both
$s$ and $t$ are. Say that a set $\Sigma$ of 
identities is \emph{realized} 
in an algebra $\m a$ 
if it is possible to interpret each function
symbol appearing in $\Sigma$ as a term of $\m a$
and each constant as an element of $\m a$
so that all identities in $\Sigma$
are satisfied by $\m a$. 
For example, every algebra $\m a$ that has underlying group structure
realizes the (basic) identities
\begin{equation}
\label{maltsevForGroups}
F(x,y,y)\approx x
\qquad\text{and}\qquad
F(y,y,x)\approx x,
\end{equation}
because these identities hold in $\m a$ for the 
group term $F(x_1,x_2,x_3)=x_1x_2^{-1}x_3$. 
A term for which the identities (\ref{maltsevForGroups}) hold in
$\m a$ is called a \emph{Maltsev term} for $\m a$.

Our paper \cite{paper1} asks the question:
Which sets $\Sigma$ of basic identities
impose a restriction on growth rates of algebras?
Phrased differently: For which sets $\Sigma$ is there
an algebra $\m a$ such that its growth rate $d_{\m a}(n)$
does not occur as the growth rate
of any algebra realizing $\Sigma$?
The answer is: exactly those
$\Sigma$ which entail the existence of a pointed
cube term. A \emph{pointed cube term} is a term
$F(x_1,\ldots,x_m)$ with respect to which $\m a$ 
satisfies an array of identities of the form
\begin{equation}\label{array}
\begin{array}{rl}
F(\wec{y}_1)&\approx x,\\
\vdots\\
F(\wec{y}_k)&\approx x,
\end{array}
\end{equation}
where each of the elements of each tuple $\wec{y}_i$
is a variable or an element of $\m a$, and a further condition
is satisfied. The condition is that, when (\ref{array})
is written as a matrix equation, $F(M) = \wec{x}$, with 
\[
M = \left(
\begin{array}{c}
\wec{y}_1\\
\vdots\\
\wec{y}_k\\
\end{array}
\right)
\quad
\textrm{and}\quad
\wec{x} = \left(
\begin{array}{c}
x\\
\vdots\\
x\\
\end{array}
\right),
\]
then each \emph{column} of $M$ contains a symbol
(a variable or constant) that is different from $x$.
The term $F$ is a 
\emph{$p$-pointed, $k$-cube term}
if the matrix $M$ contains $p$ distinct elements
of $\m a$ and $k$ rows.
Here are 
three 
basic examples: a binary term 
$F(x_1,x_2)$ for which some element $1\in A$
is a left and right unit element is a $1$-pointed,
$2$-cube term for $\m a$, 
since $\m a$ satisfies the row equations of
\[
F\left(
\begin{array}{rl}
1&x\\
x&1
\end{array}
\right)\approx 
\left(
\begin{array}{c}
x\\
x
\end{array}
\right).
\]
A 
Maltsev term $F(x_1,x_2,x_3)$ for $\m a$ is
a $0$-pointed, $2$-cube term for $\m a$, 
since 
the identities in (\ref{maltsevForGroups})
can be rewritten as the row equations of
\[
F\left(
\begin{array}{ccc}
x&y&y\\
y&y&x
\end{array}
\right)\approx 
\left(
\begin{array}{c}
x\\
x
\end{array}
\right).
\]
A 
\emph{majority term} for $\m a$ is
a $0$-pointed, $3$-cube term 
$F(x_1,x_2,x_3)$ for which 
$\m a$ satisfies the row equations of
\[
F\left(
\begin{array}{ccc}
x&x&y\\
x&y&x\\
y&x&x
\end{array}
\right)\approx 
\left(
\begin{array}{c}
x\\
x\\
x
\end{array}
\right).
\]

We prove in \cite{paper1}
that if $\Sigma$ is a set of basic identities
which entails no pointed cube term, then for any
algebra $\m a$ there is an algebra $\m b$
that realizes $\Sigma$ and has the same growth
rate as $\m a$. Thus the realization of $\Sigma$
imposes no restriction on growth rates. On the other hand, 
if $\Sigma$ entails a $p$-pointed, $k$-cube term
and $\m a$ is a (possibly infinite) algebra
for which 
$\m a^{p-1+k}$ (if $p>0$) or $\m a^k$ (if $p=0$)
is finitely generated,
then $d_{\m a}(n)$ is bounded above by a polynomial
function of $n$. This is a restriction.

In the current paper we use different techniques
to establish stronger results for algebras
with $0$-pointed cube terms, 
namely we establish that Wiegold
dichotomy holds for such algebras.
We show that if $\m a$ has a
$0$-pointed, $k$-cube term and $\m a^k$ is
finitely generated, then 
$d_{\m a}(n)=O(\log(n))$ if $\m a$ is perfect
and 
$d_{\m a}(n)=O(n)$ if $\m a$ is imperfect.
(``Big Oh'' can be strengthened to ``Big Theta''
when $\m a$ is finite.)
In this statement the word ``perfect'' is used
with respect to the modular commutator (see \cite{freese-mckenzie}),
namely an algebra is perfect if it has no nontrivial abelian
homomorphic image.

Our approach will be through an analysis
of maximal subalgebras of powers of $\m a$.
$0$-pointed cube terms were discovered
and investigated first in \cite{bimmvw}, while 
an equivalent type of term was discovered independently
and investigated in \cite{parallelogram}.
It is the results of the latter paper that are applicable
to the analysis of maximal subalgebras of powers.

\section{Preliminaries}\label{prelim_sec}

$[n]$ denotes the set $\{1,\ldots,n\}$.
A tuple in $A^n$ may be denoted 
$(a_1,\ldots,a_n)$ or $\wec{a}$.
A tuple $(a,a,\ldots,a)\in A^n$ with all coordinates equal to $a$
may be denoted $\hat{a}$.
The size of a set $A$, the 
length of a tuple $\wec{a}$, and the length of a string $\sigma$
are denoted $|A|$, $|\wec{a}|$ and $|\sigma|$.
Structures are denoted in bold face font,
e.g. $\m a$, while the universe of a structure
is denoted by the same character 
in italic font, e.g., $A$. 
The subuniverse of $\m a$ generated by a subset $G\subseteq A$
is denoted $\lb G\rb$.

We will use Big Oh notation.
If $f$ and $g$ are real-valued functions defined
on some subset of the real numbers, then 
$f\in O(g)$ and $f=O(g)$ both mean that there are 
positive constants
$M$ and $N$ such that $|f(x)|\leq M|g(x)|$ for all $x>N$.
We write $f\in \Omega(g)$ and $f=\Omega(g)$ to mean that there are 
positive constants
$M$ and $N$ such that $|f(x)|\geq M|g(x)|$ for all $x>N$.
Finally, $f\in \Theta(g)$ and $f=\Theta(g)$ mean that both
$f\in O(g)$ and $f\in \Omega(g)$ hold.

Our focus in this paper is on obtaining good
upper bounds for $d_{\m a}(n)$ 
whether $\m a$ is finite or infinite.
When $\m a$ is finite,
the upper bounds we obtain are asymptotically
equal to the easily-proved lower bounds
mentioned here:

\begin{thm}\label{first_bounds}
If $\m a$ is a finite algebra of
more than one element, then 
\begin{enumerate}
\item[(1)]
$d_{\m a}(n)\in\Omega(\log(n))$. 
\item[(2)]
$d_{\m a}(n)\in\Omega(n)$ if $\m a$ is imperfect.
\end{enumerate}
\end{thm}

\begin{proof}
Item (1) is proved in Theorem~2.2.2 of \cite{paper1}.
Item (2) follows from the combination of 
Corollary~2.2.5~(2) and Theorem~2.2.1~(2) of \cite{paper1}.
\end{proof}

We need one preliminary result for the
case when $\m a$ is infinite.

\begin{thm}\label{pointed_polynomial_cor}
If $\m a^k$ is a finitely generated algebra with a $0$-pointed
or $1$-pointed, $k$-cube term, then $d_{\m a}(n)\in O(n^{k-1})$. 
\end{thm}

\begin{proof}
This is Corollary 5.2.4 of \cite{paper1}.
\end{proof}

In particular, if $\m a$ has a $0$-pointed, 
$k$-cube term and $\m a^k$ is finitely
generated, then all finite powers
of $\m a$ are finitely generated.

Theorem~\ref{pointed_polynomial_cor}
implies that if $\m a$ has a Maltsev term
(i.e., a $0$-pointed, $2$-cube term)
and $\m a^2$ is finitely generated, then 
$d_{\m a}(n)\in O(n)$. 
We will apply this fact 
when $\m a$ is an \emph{affine} algebra (i.e., an abelian
algebra with a Maltsev term).

\section{Maximal subuniverses of powers}\label{max_sec}

In this section we relate arbitrary maximal subuniverses
of $\m a^n$ to critical maximal subuniverses. 
The results of this section 
require no assumptions on $\m a$.

\begin{df}
\label{df-induced}
If $R$ is a subuniverse of an algebra $\m b$ and $\phi\colon \m b\to\m c$
is a surjective homomorphism such that $R=\phi^{-1}(\phi(R))$, we will say 
that $R$ \emph{induced by the homomorphism} $\phi$.
\end{df}

\begin{lm}
\label{lm-induced}
If $M$ is a maximal subuniverse of $\m a^n$ and $\phi\colon \m a^n\to\m c$
is a surjective homomorphism such that $\phi(M)\not=C$, then
$\phi(M)$ is a maximal subuniverse of $\m c$ and $M$ is induced by $\phi$.
\end{lm}

\begin{proof}
If $S$ is a proper subuniverse of $\m c$ containing $\phi(M)$,
then $M\subseteq\phi^{-1}(\phi(M))\subseteq\phi^{-1}(S)\subsetneq A^n$,
where $\subsetneq$ holds, because 
$\phi(\phi^{-1}(S))=S\subsetneq C=\phi(A^n)$.
Hence the maximality of $M$ forces that $M=\phi^{-1}(\phi(M))$
and $M=\phi^{-1}(S)$. 
The first equality proves that $M$ is induced by $\phi$,
while the second equality implies that
$\phi(M)=\phi(\phi^{-1}(S))=S$, so $\phi(M)$ is a maximal
subuniverse of $\m c$. 
\end{proof}

\begin{dfs} 
\text{\cite{parallelogram}}
A \emph{compatible $n$-ary relation} of $\m a$
is a subuniverse of $\m a^n$.

A compatible relation $R$ is \emph{critical} if it is completely
$\cap$-irreducible in the subuniverse lattice of $\m a^n$ and
directly indecomposable as a relation. 
(The latter means that $R$ is not of the form 
$S\times T$ for subsets $S\subseteq A^U$
and $T\subseteq A^V$, where $\{U,V\}$ is a partition
of $[n]$ into two cells.)
\end{dfs}

Any maximal subuniverse $M$ of $\m a^n$ is 
completely $\cap$-irreducible in the subuniverse
lattice of $\m a^n$, so a
critical maximal subuniverse of $\m a^n$
is just a maximal subuniverse
that is directly indecomposable as a relation.

\begin{df}
If $M$ is a subuniverse of $\m a^n$, then a 
\emph{support} of $M$ is a subset $U\subseteq [n]$
such that $\pi_U(M)\neq A^U$,
where $\pi_U\colon \m a^n\to \m a^U$ is the projection homomorphism.
\end{df}

\begin{lm}
\label{lm-minsupp}
If $M$ is a maximal subuniverse of $\m a^n$, then $M$
has a unique minimal support. If $U$ is the 
minimal support of $M$,
then $M_U:=\pi_U(M)$ is a critical maximal subuniverse
of $\m a^U$, $M=M_U\times A^{U'}$ for $U'=[n]\setminus U$, and 
$M$ is induced by the projection $\pi_U$.
In particular, $M$ itself
is critical if and only if its 
unique support is $[n]$.
\end{lm}

\begin{proof}
For any set $U\subseteq[n]$ let $M_U:=\pi_U(M)$. 
If $U$ is a support of $M$, then
Lemma~\ref{lm-induced} applied to $\pi_U$ yields that 
$M_U$ is a maximal subuniverse of $\m a^U$,
$M$ is induced by $\pi_U$, and
hence $M=\pi_U^{-1}(M_U)=M_U\times A^{U'}$ for $U'=[n]\setminus U$.

Now assume that $U, V\subseteq [n]$ are distinct minimal supports
of the maximal subuniverse $M\leq \m a^n$.
$U$ and $V$ must be incomparable under inclusion.
We shall view elements of $\m a^n$ as functions
from $[n]$ to $A$. In this language, 
$M$ is a proper subset of the set of all functions,
$M_U$ is the set of restrictions to $U$ of the functions
in $M$, and $M$ contains all functions whose restriction to
$U$ belongs to $M_U$. Similarly,
$M_V$ is the set of restrictions to $V$ of the functions
in $M$, and $M$ contains all functions whose restriction to
$V$ belongs to $M_V$.
Since $M_U\neq A^U$, there is a function $f\colon U\to A$ 
that is not in $M_U$. Since $V$ is a minimal support
and $U\cap V$ is properly contained in $V$, it follows that every 
function $U\cap V\to A$ is the restriction
of some function in $M$. In particular, 
$f|_{U\cap V}=g|_{U\cap V}$ for some $g\in M$.
Let $h\in A^n$ be any function that agrees with
$f$ on $U$ and $g$ on $V$. 
Then $h|_U = f\notin M_U$, so $h\notin M$.
Yet $h|_V = g|_V\in M_V$, so $h\in M$, a contradiction.
This shows that $M$ has a unique minimal support.

Let $U$ be the  minimal support of $M$.
The second statement of the lemma,
except for the criticality of $M_U$,
follows from the first paragraph
of this proof.
To show that $M_U$ is a critical, assume that $M_U = S\times T$, where
$S\leq \m a^X$ and $T\leq \m a^Y$ for some partition
$\{X,Y\}$ of $U$. Since $M_U\not=A^U$,
either 
$A^X\neq S=\pi^{\m a^U}_X(M_U) = \pi^{\m a^n}_X(M)$, or 
$A^Y\neq T=\pi^{\m a^U}_Y(M_U) = \pi^{\m a^n}_Y(M)$.
Either way, one obtains that $X$ or $Y$ is a proper
subset of $U$ that is a support of $M$, contradicting
the minimality of $U$.

For the final statement of the lemma, if the minimal support of 
$M$ is $[n]$, then $\pi_{[n]}(M)=M$ is critical by the second
statement of the lemma.
Conversely, assume that $M$ is critical and $U\subseteq [n]$
is its minimal support. Since $M = M_U\times A^{U'}$
and $M$ is directly indecomposable
as a relation, we get $U' = \emptyset$, equivalently $[n]=U$.
\end{proof}

\section{The parallelogram property for critical relations}\label{parallel}

In the preceding section we showed that all maximal subuniverses
of $\m a^n$ are induced by critical maximal subuniverses
on projections $\m a^U$ of $\m a^n$. In this section
we show that the critical maximal subuniverses of $\m a^U$
have a special structure when $\m a$ has a $0$-pointed,
$k$-cube term.

\begin{df} 
\text{\cite{parallelogram}}
Given a partition $\{S, T\}$ of $[n]$ into two cells,
write $\wec{x}\wec{y}$ for a tuple in $A^n$ to mean that 
$\wec{x}\in A^S$ and $\wec{y}\in A^T$.
A compatible $n$-ary relation $R$ satisfies the 
\emph{parallelogram property} if, for any partition
$\{S, T\}$ of $[n]$, 
$\wec{a}\wec{u}, \wec{a}\wec{v}, \wec{b}\wec{v}\in R$
implies $\wec{b}\wec{u}\in R$.
\end{df}

Theorem~3.5 and Theorem~3.6~(3) of \cite{parallelogram} together prove the
following theorem.

\begin{thm}
A variety $\mathcal V$ has a $0$-pointed, $k$-cube term
if and only if every member $\m a\in \mathcal V$
has the property that any critical relation of $\m a$
of arity at least $k$ has the parallelogram property.
\end{thm}

It follows from this theorem and Lemma~\ref{lm-minsupp}
that if $\m a$ has 
a $0$-pointed, $k$-cube term, and $M$ is a maximal
subuniverse of $\m a^n$ 
with minimal support $U\subseteq[n]$, then
$M=M_U\times A^{U'}$ ($U'=[n]\setminus U$), and
either
$M_U$ has arity less than $k$ or 
$M_U$ is a critical maximal subuniverse of $\m a^U$ 
that has the parallelogram property.
(In the latter case, $M$ itself will also have the parallelogram
property.) Our next step is to investigate the structure
of maximal subuniverses with the parallelogram property.

The paper \cite{parallelogram} analyzes arbitrary
compatible relations with the parallelogram property
in congruence modular varieties. It is shown in
\cite{bimmvw} that any algebra with a $0$-pointed,
$k$-cube term generates a congruence modular
variety, so the results of \cite{parallelogram} apply here.
The first step in the analysis is the ``reduction''
of a relation, which we describe next.

Suppose that $R\leq \m a^n$ is a compatible relation with
the parallelogram property; as a special case, suppose
that $M\leq \m a^n$ is a maximal critical subuniverse 
with the parallelogram property.
For the first step in the reduction, 
realize $R$ as a subdirect product
$R\leq_{\text{sd}} \prod_{i=1}^n \m a_i$, 
where $\m a_i:= \pi_i(R)\leq \m a$.
In the special case involving the maximal subuniverse
$M$ we will have 
$\m a_i= \pi_i(M) =  \m a$ unless the projection of $M$
onto one single coordinate is not surjective.
This happens only if $M$ has a support of size one,
which, by criticality, implies that $M$ is a unary relation.
We henceforth consider only $M$
of arity at least two, so that in our special case 
$\pi_i(M)=A$ for all $i$. Thus, in the first step in reduction,
nothing happens if $M$ is maximal and of arity greater than one.

Second, define relations, called \emph{coordinate kernels}
in \cite{parallelogram}, 
\[
\theta_i = \{(a,b)\in A_i^2\;|\;\exists \wec{c}\in \prod_{j\neq i}\m a_j\;
(a\wec{c}\in R\;\&\;b\wec{c}\in R)\}.
\]
It is proved in Lemma~2.3 of \cite{parallelogram} that
(i) each $\theta_i$ is a congruence on $\m a_i$, and
(ii) $R$ is induced by the homomorphism 
$\psi\colon \prod \m a_i\to \prod \m a_i/{\theta_i}$ that is the natural
map in each coordinate.
The relation $\overline{R}=\psi(R)$ is the \emph{reduction} of $R$.

In our special case $M\leq \m a^n$ 
is critical and
maximal, 
therefore by Lemma~\ref{lm-induced}
its reduction
$\overline{M} = \psi(M)$ is a maximal
subuniverse of $\prod \m a/{\theta_i}$. 

The next result is a specialization of 
(some parts of) Theorem~2.5 of \cite{parallelogram}
to the case where $M$ is a critical maximal subuniverse of $\m a^n$
and $n>1$. We maintain the numbering
of \cite{parallelogram}, but omit the unused parts
of the theorem.

\begin{thm}\label{para_main}
Let $M$ be a critical maximal subuniverse of $\m a^n$
that satisfies the parallelogram property,
and let $\overline{M}\leq \prod \m a/{\theta_i}$
be its reduction. If
$n>1$ and $\m a$ lies in a congruence modular variety, 
then the following hold.
\begin{enumerate}
\item[$(1)\hphantom{^*}$] 
$\overline{M}\leq \prod \m a/{\theta_i}$ is a representation 
of $\overline{M}$ as a subdirect product of subdirectly
irreducible algebras.
\item[$(5)^*$]
If $n>2$, then the monolith of $\m a/{\theta_i}$
is the total relation; i.e. $\m a/{\theta_i}$ is simple.
\item[$(7)^*$] 
If $n>2$, then each simple algebra 
$\m a/{\theta_i}$ is abelian.
\end{enumerate}
\end{thm}

Here, items 
(5) and (7) are marked with asterisks, because
we have altered the statement of (5) from \cite{parallelogram}
in order to take into account that $\overline{M}$
is a \emph{maximal} subuniverse of $\prod \m a/{\theta_i}$ 
and we have altered the statement of (7)
in order 
to take into account the
conclusion from $(5)^*$ that 
$\m a/{\theta_i}$ is simple.

We explain what this theorem contributes to our current investigation.
Suppose that $\m a$ has a $0$-pointed,
$k$-cube term. Suppose also that
$M\leq \m a^n$ is maximal, $U$
is the minimal support of $M$,
and $M=M_U\times A^{U'}$ is induced by $\pi_U\colon \m a^n\to \m a^U$.
If $|U|$ is at least as large as $\max\{3, k\}$,
then the theorem proves that 
$M_U$ is induced by a homomorphism 
$\psi\colon \m a^U\to \prod_U \m a/{\theta_i}$ where 
each factor $\m a/{\theta_i}$ is a simple abelian algebra.
Thus, $M$ itself is induced by the composition
of the surjective homomorphisms
\[
\m a^n 
\stackrel{\pi_U}{\longrightarrow} 
\m a^U
{\longrightarrow} 
\left(\m a/{[1,1]}\right)^U
\longrightarrow
\prod_U \m a/{\theta_i},
\]
where the last two maps are a factorization
of the map 
$\psi\colon \m a^U\to \prod_U \m a/{\theta_i}$ which induces $M_U$,
and these two maps are defined
coordinatewise by the natural maps
$\m a\to \m a/{[1,1]}\to \m a/\theta_i$.
(We have $\theta_i\geq [1,1]$, since $\m a/{\theta_i}$ is abelian.)
Hence $M$ is induced by the sub-composition
$\m a^n 
\stackrel{\pi_U}{\longrightarrow} 
\m a^U
{\longrightarrow} 
\left(\m a/{[1,1]}\right)^U$, which may be factored another way
as 
$\m a^n 
\stackrel{\eta}{\longrightarrow} 
\left(\m a/{[1,1]}\right)^n
\stackrel{\pi_U}{\longrightarrow} 
\left(\m a/{[1,1]}\right)^U$.
Hence $M$ is induced by the single map
$\eta$, which maps $\m a^n$ onto its abelianization.
Altogether this proves the desired result:

\begin{thm}\label{induced}
Assume that $\m a$ has a $0$-pointed, $k$-cube term.
If $M\leq \m a^n$ is a maximal subuniverse, then either
\begin{enumerate}
\item[($\pi$)] 
$M$ is induced by a projection $\pi_U\colon \m a^n\to \m a^U$
for some subset $U\subseteq [n]$ satisfying $|U| < \max\{3,k\}$, or 
\item[($\eta$)]
$M$ is induced by $\eta\colon \m a^n\to (\m a/{[1,1]})^n$.
\end{enumerate}
\end{thm}

\section{A solution to a combinatorial problem}\label{solution}

To derive our result
on growth rates from Theorem~\ref{induced},
we will use a solution to the following problem:
If $B$ is a finite set and $n\geq k>1$ are 
integers, then how small can a set $G\subseteq B^n$ be
if its projection onto any subset of $k$ coordinates
is surjective?

If $B$ is finite, $G\subseteq B^n$ and $|G|=g$, 
then $G$ can be linearly ordered
and taken to be the sequence of 
rows of a $g\times n$ matrix of elements
of $B$, say $[b_{i,j}]$. If 
\[
\sigma\colon \quad 1 \leq j(1) < \cdots < j(k) \leq n
\]
is a selection of $k$ numbers between $1$ and $n$, then 
the projection of $G$ onto the
coordinates in $\sigma$ 
is the set of row vectors
$(b_{1, j(1)}, \ldots, b_{1, j(k)}), \ldots, 
(b_{g, j(1)}, \ldots, b_{g, j(k)})$ which 
occur as the set of rows of the $g\times k$
minor of $[b_{i,j}]$ whose column indices
are the indices in $\sigma$.
$G$ projects surjectively onto each $k$ coordinates of $B^n$ 
if and only if, for each choice $\sigma$
of $k$ column indices, the set of row vectors 
of the corresponding $g\times k$ minor of $[b_{i,j}]$ exhausts $B^k$.
Therefore, call a $g\times k$ matrix of elements of $B$
a \emph{bad minor} (or \emph{bad matrix})
if its rows fail to exhaust $B^k$.
The desired property of $G$ is that its associated
matrix has no bad minors.

\begin{thm}\label{surj_projections}
Let $B$ be a finite set of size $|B|=b > 1$.
Let $n\geq k > 1$ be natural numbers, and set
$\displaystyle u = b^k/(b^k-1)$.
If 
$
g\geq k\log_u\left(n\right) + \log_u\left(b^k/k!\right),
$
then there is a matrix in $B^{g\times n}$ with no bad minors.
\end{thm}

\begin{proof}
This is a probabilistic proof.
Our sample space is the set $B^{g\times n}$ of all $g\times n$
matrices of elements of $B$. 
Our probability distribution is the uniform one, 
so each individual matrix $M\in B^{g\times n}$ 
has probability $P(M)=|B^{g\times n}|^{-1} = b^{-gn}$.
For each matrix $M\in B^{g\times n}$ and 
each sequence of $k$ column indices, 
\[
\sigma\colon \quad 1 \leq j(1) < \cdots < j(k) \leq n,
\]
let $M_{\sigma}$ denote the $g\times k$ minor of $M$
whose column indices are those enumerated by $\sigma$
(called the \emph{$\sigma$-minor} of $M$).
Let $X_{\sigma}$ be the random variable whose value 
at the element $M\in B^{g\times n}$ is $1$ if 
$M_{\sigma}$ is a bad minor and $0$ otherwise,
i.e., $X_{\sigma}$ is the indicator variable for bad $\sigma$-minors.

\begin{clm}\label{expectation_clm}
For any $\sigma$, 
the expected value of $X_{\sigma}$ satisfies
\begin{equation}\label{first_estimate}
E(X_{\sigma})\leq b^k\left(b^k-1\right)^gb^{-gk}.
\end{equation}
\end{clm}

The expectation is computed 
\[
\begin{array}{rl}
E(X_{\sigma}) 
&=\; \sum_{M\in B^{g\times n}} (X_{\sigma}(M)\cdot P(M))\\
&=\; \sum_{M\in B^{g\times n}} (X_{\sigma}(M)\cdot b^{-gn})\\
&=\; \left(\sum_{M\in B^{g\times n}} X_{\sigma}(M)\right)b^{-gn},
\end{array}
\]
where the sum $\sum_{M\in B^{g\times n}} X_{\sigma}(M)$
on the last line represents the number matrices
in $B^{g\times n}$ whose $\sigma$-minor is bad.
By definition, a $g\times k$ matrix is bad if some
tuple $\wec{b}\in B^k$ does not appear among its rows.
So, 
for each $\wec{b}\in B^k$, let $\mathcal U_{\wec{b}}$
denote the set of all $g\times k$ matrices where
$\wec{b}$ does not appear among the rows.
$|\mathcal U_{\wec{b}}|$ can be computed by noting that
the $g$ rows of a matrix in $\mathcal U_{\wec{b}}$
may be freely chosen
from the set $B^k-\{\wec{b}\}$, which has
size $b^k-1$, so $|\mathcal U_{\wec{b}}| = (b^k-1)^g$.
The bad $g\times k$ matrices are those from
$\bigcup_{\wec{b}\in B^k} \mathcal U_{\wec{b}}$. Since
the cardinality of the union is no more than the sum of the
individual cardinalities, and these summands have the
same size,  we get that the number of bad
$g\times k$ matrices is no more than
$|B^k|\cdot |\mathcal U_{\wec{b}}| = b^k(b^k-1)^g$.
Each bad $g\times k$ matrix $N$ can be extended in 
$b^{g(n-k)}$ ways to a matrix $M\in B^{g\times n}$ whose
$\sigma$-minor satisfies $M_{\sigma}=N$, so 
the number of matrices in $B^{g\times n}$ with
a bad $\sigma$-minor is no more than 
$b^k(b^k-1)^gb^{g(n-k)}$.
Hence 
\[
E(X_{\sigma}) = 
\left(\sum_{M\in B^{g\times n}} X_{\sigma}(M)\right)b^{-gn}
\leq b^k(b^k-1)^g b^{g(n-k)} b^{-gn} = b^k(b^k-1)^g b^{-gk},
\]
as claimed.

\medskip

If $X:=\sum_{\sigma} X_{\sigma}$ is the sum of all
$X_{\sigma}$ as $\sigma$ ranges
over all $\binom{n}{k}$ choices of $k$ column indices
and $M\in B^{g\times n}$, then $X(M)$
equals the number of bad $g\times k$ minors of $M$. Since 
expectation is linear, and since 
$\binom{n}{k} < n^k/k!$ when $n\geq k > 1$, we get 
from (\ref{first_estimate})
that 
\[
E(X) = \sum_{\sigma} E(X_{\sigma}) \leq \binom{n}{k} b^k(b^k-1)^g b^{-gk} <
n^k(b^k/k!)(b^k-1)^g b^{-gk}.
\]
If it is the case that 
\begin{equation}\label{bound}
n^k(b^k/k!)(b^k-1)^g b^{-gk}\leq 1,
\end{equation}
then we will have $E(X) < 1$, meaning that the expected number of 
bad minors in an element of $B^{g\times n}$
is strictly less than $1$. This can happen only
if matrices without bad minors exist.
Rewriting (\ref{bound}) as
\[
n^k \leq 
\left(\frac{b^{k}}{\left(b^k-1\right)}\right)^g(b^k/k!)^{-1} = 
u^g(b^k/k!)^{-1},
\]
using the definition $u = b^k/(b^k-1)$, we can solve for $g$
to get
\begin{equation}\label{4rth}
g\geq k\log_u(n) + \log_u(b^k/k!).
\end{equation}
When this inequality holds
we get that (\ref{bound}) holds,
so a matrix with no bad minors exists. This is exactly the statement
of the theorem.
\end{proof}

\begin{cor}\label{surj_projections_cor}
Let $B$ be a finite set of size $|B|=b > 1$.
Let $n\geq k > 1$ be natural numbers, and set
$u = b^k/(b^k-1)$.
If 
$
g=\lceil k\log_u\left(n\right) + \log_u\left(b^k/k!\right)\rceil,
$
then there exists a subset $G\subseteq B^n$ of size $g$
whose projection onto any
$k$ coordinates of $B^n$ is surjective.
\end{cor}

\begin{proof}
By the theorem, there is a matrix in $B^{g\times n}$ with no bad minors.
The set $G\subseteq B^n$ consisting of the 
rows of this matrix has size $g$ and projects surjectively onto any
$k$ coordinates of $B^n$.
\end{proof}

\begin{cor}\label{surj_projections_cor_2}
Let $\m a$ be an algebra, and suppose
that for some $k>1$ the algebra
$\m a^k$ is generated by a finite
set $H\subseteq A^k$.
Let $B\subseteq A$ be the finite set of elements of that
appear in the coordinates of tuples in $H$.
Set $b = |B|$ and $u = b^k/(b^k-1)$.
For any $n\geq k$, if 
$
g=\lceil k\log_u\left(n\right) + \log_u\left(b^k/k!\right)\rceil,
$
then there exists a subset $G\subseteq B^n\subseteq A^n$ of size $g$
such that the subalgebra $\m s = \langle G\rangle\leq \m a^n$
has the property that the projection of $\m s$ onto any
$k$ coordinates of $\m a^n$ is surjective.
\end{cor}

\begin{proof}
Here we choose $G\subseteq B^n$ as in Corollary~\ref{surj_projections_cor}
so that the projection of $G$ onto any $k$ coordinates of $B^n$
is surjective. Any projection 
$\pi_U\colon \m s = \langle G\rangle\to \m a^U$
of the subalgebra $\m s\leq \m a^n$
onto a $k$-element set $U\subset [n]$
contains the projection of 
the subset $B^n\subseteq S$
onto those $k$ coordinates, and $\pi_U(B^n)=B^U$
is a generating
set for $\m a^U$. Hence $\m s\leq \m a^n$
has the property that the projection onto any
$k$ coordinates of $\m a^n$ is surjective.
\end{proof}

\section{Growth rates of algebras with a cube term}\label{growth_cube}

In this section we combine the preceding 
results to obtain the following.

\begin{thm}\label{growth_cube_thm}
Suppose that $\m a$ has a $0$-pointed, $k$-cube term
and that $\m a^k$ is finitely generated.
If $\m a$ perfect, then 
$d_{\m a}(n)\in O(\log(n))$.
If $\m a$ is imperfect, then
$d_{\m a}(n)\in O(n)$.
\end{thm}

\begin{proof}
According to Theorem~\ref{pointed_polynomial_cor}, 
the fact that $\m a^k$ is finitely generated
implies that $\m a^n$ is finitely generated for all finite $n$.
Hence any proper subuniverse of $\m a^n$
is contained in a maximal subuniverse of $\m a^n$.

According to Theorem~\ref{induced}, 
if $M\leq \m a^n$ is a maximal subuniverse, then either
\begin{enumerate}
\item[($\pi$)] 
$M$ is induced by a projection $\pi_U\colon \m a^n\to \m a^U$
for some  subset $U\subseteq [n]$ satisfying $|U| < \max\{3,k\}$, or 
\item[($\eta$)]
$M$ is induced by $\eta\colon \m a^n\to (\m a/{[1,1]})^n$.
\end{enumerate}

For each $n$, choose a subset $G_{\pi}\subseteq A^n$ of size
$O(\log(n))$ 
such that
the subalgebra $\langle G_{\pi}\rangle$ of
$\m a^n$ 
has the property that its projection onto any
$\max\{3,k\}$ coordinates of $\m a^n$ is surjective.
The existence of such a set
is guaranteed by Corollary~\ref{surj_projections_cor_2}.
Clearly $G_{\pi}$ is contained in no maximal subuniverse
of $\m a^n$ that is induced by a projection
onto any subset of less than $\max\{3,k\}$ coordinates.

The algebra $\m a/{[1,1]}$ is abelian and has a cube
term, so $\m a/{[1,1]}$ is affine by \cite[Corollary~5.9]{freese-mckenzie}.
According to the remarks following
Theorem~\ref{pointed_polynomial_cor}, 
$(\m a/{[1,1]})^n$ contains a set of generators
of size $O(n)$. For each $n$, choose a set 
$G_{\eta}\subseteq A^n$ of size $O(n)$
such that $\eta(G_{\eta})$ generates 
$(\m a/{[1,1]})^n$. Then $G_{\eta}$ is contained
in no maximal subuniverse of $\m a^n$ induced by $\eta$.

We now have that $G_{\pi}\cup G_{\eta}$ is a set
of size $O(n)$ that is contained
in no maximal subuniverse of $\m a^n$,
hence $G_{\pi}\cup G_{\eta}$ is a generating set for $\m a^n$
of size $O(n)$.

When $\m a$ is perfect, then $\m a^n$ has no 
maximal subuniverses induced by $\eta$,
so $G_{\pi}$ is a generating set for $\m a^n$
of size $O(\log(n))$.
\end{proof}

\begin{cor}
Suppose that $\m a$ is finite and has
a $0$-pointed, $k$-cube term.
If $\m a$ perfect, then 
$d_{\m a}(n)\in \Theta(\log(n))$.
If $\m a$ is imperfect, then
$d_{\m a}(n)\in \Theta(n)$.
\end{cor}

\begin{proof}
Combine the upper bounds of
Theorem~\ref{growth_cube_thm}
with the lower bounds of Theorem~\ref{first_bounds}.
\end{proof}

\bibliographystyle{plain}

\begin{thebibliography}{10}



\bibitem{bimmvw}
Berman, Joel, Idziak, Pawe{\l}, Markovi\'c, Petar, McKenzie, Ralph, 
Valeriote, Matthew, Willard, Ross, 
{\it Varieties with few subalgebras of powers.}
Trans.\ Amer.\ Math.\ Soc.\  {\bf 362}  (2010),  no. 3, 1445--1473.



\bibitem{freese-mckenzie}
Freese, Ralph, McKenzie, Ralph,
{\sl Commutator Theory for Congruence
Modular Varieties},
London Mathematical Society Lecture Note Series {\bf 125}, 
Cambridge University Press, Cambridge, 1987.



\bibitem{paper1}
Kearnes, Keith, Kiss, Emil, Szendrei, \'Agnes,
{\it Growth rates of finite algebras, I: pointed cube terms},
manuscript.

\bibitem{paper3}
Kearnes, Keith, Kiss, Emil, Szendrei, \'Agnes,
{\it Growth rates of finite algebras, III: solvable algebras},
manuscript.

\bibitem{parallelogram}
Kearnes, Keith, Szendrei, \'Agnes,
{\it Clones of algebras with parallelogram terms.}
Internat.\ J.\ Algebra Comput.\ 
{\bf 22} (2012), no. 1, 1250005, 30 pp.




\bibitem{quick_ruskuc}
Quick, Martyn, Ru\v skuc, Nik, 
{\it Growth of generating sets for direct powers 
of classical algebraic structures.} 
J.\ Austral.\ Math.\ Soc.\ {\bf 89}  (2010), 105--126.



\bibitem{wiegold1}
Wiegold, James, 
{\it Growth sequences of finite groups. }
Collection of articles dedicated to the memory of Hanna Neumann, VI. 
J.\ Austral.\ Math.\ Soc.\ {\bf 17} (1974), 133--141.

\end{thebibliography}

\end{document}